\documentclass[12pt,a4paper]{amsart}
\usepackage[cp1250]{inputenc}
\usepackage{amsthm}
\usepackage{amssymb,amsmath}
\usepackage{amssymb}
\usepackage{graphicx}
\usepackage{mathrsfs}
\usepackage{color}
\renewcommand{\leq}{\leqslant}
\renewcommand{\geq}{\geqslant}

\newcommand{\rr}{\mathbb R}
\newcommand{\nn}{\mathbb N}
\newcommand{\vf}{\varphi}

\newcommand{\argmin}{\operatorname{argmin}}
\numberwithin{equation}{section}

\newcounter{thm}[section]

\newtheorem{tw}[thm]{Theorem}

\newtheorem{rem}[thm]{Remark}

\newtheorem{lemat}[thm]{Lemma}
\newtheorem{cor}[thm]{Corollary}

\newtheorem{exa}[thm]{Example}

\theoremstyle{definition}

\begin{document}

\title[Exponential convexifing of polynomials]{Exponential convexifying  of polynomials }
 \keywords{Polynomial, semialgebraic set, convex function, exponential function, lower critical point.}
  \subjclass[2010]{Primary 11E25, 12D15; Secondary 26B25.}
\date{\today}

\author[K. Kurdyka]{Krzysztof Kurdyka} 
\address{Krzysztof Kurdyka, Laboratoire de Mathematiques (LAMA) Universit\'e Savoie Mont Blanc,  UMR-5127 de CNRS 
73-376 Le Bourget-du-Lac cedex FRANCE}
\email{Krzysztof.Kurdyka@univ-savoie.fr}

\author[K. Kuta]{Katarzyna Kuta}
\address{Katarzyna Kuta, Faculty of Mathematics and Computer Science, University of \L \'od\'z, 
S. Banacha 22, 90-238 \L \'od\'z, POLAND}
\email{kkuta@math.uni.lodz.pl}

\author[S. Spodzieja]{Stanis{\l}aw Spodzieja}
\address{Stanis{\l}aw Spodzieja, Faculty of Mathematics and Computer Science, University of \L \'od\'z, 
S. Banacha 22, 90-238 \L \'od\'z, POLAND}
\email{spodziej@math.uni.lodz.pl}



\begin{abstract} 

Let $X\subset\rr^n$ be a convex closed and semialgebraic set and let $f$ be a polynomial
 positive on $X$. We prove that there exists an exponent  $N\geq 1$, such that for any $\xi\in\rr^n$ the function $\varphi_N(x)=e^{N|x-\xi|^2}f(x)$ is  strongly convex on $X$. When $X$ is unbounded  we have to assume also that the leading form of $f$ is positive in $\rr^n\setminus\{0\}$.  We obtain   strong convexity of $\varPhi_N(x)=e^{e^{N|x|^2}}f(x)$ on possibly unbounded $X$, provided $N$ is sufficiently large,    assuming only  that  $f$ is positive on $X$. We apply these results for searching critical points of polynomials on convex closed semialgebraic sets. 
\end{abstract}
\maketitle

\section{Introduction}


In \cite{KS} we considered several questions  concerning convexification of a polynomial $f$
  which is positive on a closed convex set $X\subset \mathbb R^n$.
  One of the main results in  \cite{KS},  is the following [Theorem 5.1]: \emph{if $X$ is a compact set than there exists a positive integer $N$  such that the function 
\begin{equation}\label{conv1}
\phi_N(x)=(1+|x|^2)^Nf(x)
\end{equation}
is strongly convex on $X$}. Moreover,  explicit estimates for the exponent $N$ were given in  \cite{KS}. They depend on the diameter of $X$, the size of coefficients of the polynomial $f$ and on the minimum of $f$ on $X$. 
In fact a  stronger version of  \eqref{conv1} was  given  in  \cite{KS}; \emph{there exists  an integer  $N$, which can be explicitly estimated, such  that the polynomials
\begin{equation*}
\phi_{N,\xi}=(1+|x-\xi|^2)^Nf(x),\quad \xi\in X,
\end{equation*}
are strongly convex on $X$}. The fact that $N$ can be chosen  independent  of  $\xi$  was  crucial for  a construction
of an algorithm which for a given polynomial $f$, positive in the convex compact semialgebraic set $X$, produces a sequence $a_\nu\in X$ starting from an arbitrary point $a_0\in X$, defined by induction: $a_\nu=\argmin_{X}\phi_{N,a_{\nu-1}}$, i.e., $a_{\nu}\in X$ is the unique point of $X$ at which $\phi_{N,a_{\nu-1}}$ has a global minimum on $X$. The sequence $a_\nu$ converges to a lower critical point of $f$ on  $X$ (see \cite[Theorem 7.5]{KS}).

In the case of non-compact closed convex set $X$ the  results mentioned   above require an  additional assumption, that the \emph{leading form} $f_d$ of  $f$, 
satisfy  
\begin{equation}\label{eqpositivityfd}
f_d(x)>0\quad\hbox{for }x\in\rr^n\setminus\{0\}.
\end{equation}
Under this assumption we have that: \emph{if a  polynomial $f$ is positive on $X$ then for any $R>0$ there exists $N_0$ such that for each $\xi\in X$, $|\xi|\leq R$,
$N>N_0$ the polynomial $\phi_{N,\xi}$ is strongly convex on $X$}.

The assumption \eqref{eqpositivityfd} is necessary for local convexity of $\phi_{N,\xi}$ in a neighborhood of infinity, see \cite[Proposition 6.3]{KS}. However, this assumption is not sufficient to obtain  convexity of the polynomial $\phi_{N,\xi}$ for some fixed  $N>0$ independent of  $\xi\in X$. 
For instance 
the polynomial $f(x)=1+x^2$, has this property, cf. \cite[Example 4.5]{KS}. 
%
%


The  main goal  of this paper is to study convexification of polynomials functions  by exponential factors of the form
$e^{N|x-\xi|^2}$ or by double exponential of the form $e^{e^{N|x-\xi|^{2}}}$.
 Surprisingly they  play distinct roles.
We set 
 $$
\varphi _{N,\xi}(x):=e^{N|x-\xi|^2}f(x).
$$ 
and  prove the following  (see Theorem \ref{uwypuklanie na zwartym} and Corollary \ref{uwypuklanie na zwartymcor2}): 
\emph{if  a polynomial $f$ is positive on a compact and convex set $X\subset \mathbb R^n$, 
than there exists effectively computed number $N_0$ such that for any $N> N_0$ and $\xi\in \rr^n$ the function $\varphi _{N,\xi}(x)$ 
 is strongly convex on $X$.} 


If  $X$ is not  compact, we obtain the above assertions under the assumption 
 \eqref{eqpositivityfd}, see Theorem \ref{convexonconvexsetnewvar}. 
  In general the assumption \eqref{eqpositivityfd} 
  can not be ommited as we show in Example \ref{exacounter}. 
 
Surprisingly convexification in the noncompact case  without assumption \eqref{eqpositivityfd}  is possible using double exponential factors. Namely  in Theorems \ref{tedoubleexp} and \ref{convexonconvexsetnewxdoublevarsemi} we prove that:
\emph{if $X\subset \rr^n$ is a convex and closed semialgebraic set and $f$ is a polynomial positive on $X$, then for any $R>0$ there exists effectively computed number $N_0$ such that for any $N> N_0$ and any $\xi\in \rr^n$, $|\xi|\leq R$, the function 
$$
\varPhi_{N,\xi}(x):=e^{e^{N|x-\xi|^{2}}}f(x)
$$ 
is strongly convex on $X$.} 
 
In the case when $X$ is a convex and closed set, but non necessary semialgebraic, the result  still holds  (Theorems \ref{tedoubleexp2} and \ref{convexonconvexsetnewxdouble}) under an additional assumption 
\begin{equation}\label{eqestmonX}
\inf\{f(x): x\in X\}\ge  m>0.
\end{equation}

In the above    theorems  one can  replace $\varPhi_{N,\xi}$ by  the function 
$$
\Phi_{N,\xi}(x):=e^{Ne^{|x-\xi|^{2}}}f(x).
$$
 
 It turns out that convexification of polynomials using exponential function  is somehow more natural and powerful than the convexification   by the  factors of the form 
 $(1+|x-\xi|^2)^N$ done in \cite{KS}. In particular it applies also to the noncompact case  and the explicit formulae for the exponent $N$ are nicer.
 
  We  believe that the results mentioned above could be of    interest, also to  study   o-minimal structures expanded by the exponent function. It fits particularly to the structure 
 $\rr_{\exp}$  semialgebraic sets expanded by the  exponent function. The remarkable fact that $\rr_{\exp}$ is indeed an o-minimal structure
 was established by A. Wilkie \cite{W}.
 It would be interesting  to explain
  a different power of exponential and double exponential  for convexification.

The main difficulty when determining explicitely the number $N$ such that the function $\varphi_N$ is strongly convex on a convex  compact set $X$, comes from an effective estimation of the number $m$ in \eqref{eqestmonX} 
and the number $R=\max\{|x|:x\in X\}$. Using results of G.~Jeronimo, D. Perrucci, E.~Tsigaridas \cite{JPT}  we show in  Theorem \ref{uwypuklaniecalkowitychna zwartym}) and Theorem \ref{uwypuklanie na zwartym} how it is feasible when $X$ is a compact semialgebraic set described by polynomial inequalities with integer coefficients and $f$ is also a polynomial with integer coefficients (see Theorem \ref{uwypuklanie na zwartym calk}).

As an application to optimization  we propose an algorithm which produces, starting from an arbitrary point $a_0\in X$,
  a sequence $a_\nu \in X$ which tends to a lower critical point of a polynomial $f$ restricted to 
  $X$ or to infinity. We assume that 
 $X\subset \rr^n$ is a closed convex  semialgebraic set and $f$ a polynomial  which is bounded from below on $X$.  Then by adding to $f$ an appropriate constant we may assume that  $f\ge m> 0$ on $X$. If $X$ is unbounded we assume also condition \eqref{eqpositivityfd}.
Hence by the above mentioned theorems we  obtain strong convexity of $\varphi_\xi(x)=e^{|x-\xi|^2}f(x)$ for $\xi\in X$.
Let us choose any  $a_0\in X$ and set by induction: 
$a_\nu=\argmin_{X}\varphi_{N,a_{\nu-1}}$.  Then we prove that  the sequence  $a_\nu $   tends to a lower critical point of a polynomial $f$ restricted to 
  $X$ or to infinity.  Note that computing $a_\nu$, that is minimizing $\varphi_{N,a_{\nu-1}}$  on $X$,  is usually easier since the function is convex.
   This type of algorithm, based on convexification,  is called  sometimes {\it proximal}, see for instance \cite{Bolte}. 
  Observe  that computing the critical point of $\varphi_{N,a_{\nu-1}}$ involves only algebraic equations. 
%

The paper is organized as follows. In Section \ref{one variable} we prove that the function $\varphi _N$ in one variable  is strongly convex on a closed interval $I\subset \mathbb R$, provided $f(x)>m$ for $x\in I$ and some $m>0$ and $N\in\mathbb{R}$ is sufficiently large. We also estimate from above the number $N$.  In Section \ref{Sect3} we consider this problem in  the several variables case on a compact convex set $X$ (see Theorem \ref{uwypuklanie na zwartym}). In  Sections \ref{Sect44} and \ref{secdoubleexp} we consider the case when the set $X$ is not compact. 

\section{Convexifying polynomials}\label{one variable}
\subsection{Convexifying $C^2$-functions in one variable}
In this section we prove that if $f$ is a function of class $C^2$ positive on a closed interval $I\subset \mathbb R$ (not necessary compact), then for $N$ large enough the function $t\mapsto e^{Nt^2}f(t)$ is strongly convex on   $I$.

Let $f: \mathbb R\rightarrow \mathbb R$ be $C^2$ function. For any $N\in \mathbb R $ and $p,q\in\rr$ we define the following function:
$$
\varphi _{N,p,q}(t):=e^{N(t^2+pt+q)}f(t), \ \ t\in \mathbb R.
$$

For positive numbers $m,D$ we put
$$ 
\mathcal{N}(m,D):=\frac{D}{2m}+\frac{D^2}{2m^2}.
$$

\begin{lemat}\label{lemat dla funkcji2}
Let $f$ be a function of class $C^2$, which is positive on a closed interval $I\subset \mathbb R$. Let 
$m,D\in\mathbb{R}$ be such that 
\begin{equation}\label{minimum f}
	0<m\leq \inf\{f(t): t \in I\},
\end{equation}
and 
\begin{equation}\label{ograniczenie pochodnych}
|f'(t)|\leq D, \quad |f''(t)|\leq D \quad for \quad t \in I.
\end{equation} 
Assume that 
 $p^2\leq 4q$ and
\begin{equation}\label{nierownosc dla N2}
N>\mathcal{N}(m,D)
\end{equation} 
then
$$\varphi _{N,p,q}'' (t)\geq -\frac{D^2}{m}-D+2Nm>0$$ for $t \in I$, thus $\varphi _{N,p,q}$ is strongly convex on $I$.
\end{lemat}

\begin{proof}
By definition of $\varphi_{N,p,q}$ we have
\begin{equation*}\label{eqwaznewprzykladzie}
\varphi_{N,p,q}''(t)=e^{N(t^2+pt+q)}[N^2(2t+p)^2f(t)+2N(2t+p)f'(t)+2Nf(t)+f''(t)].
\end{equation*}
Hence, from the assumptions we obtain
\begin{equation}\label{oszacowanie}
\varphi_{N,p,q}''(t)\geq e^{N(t^2+pt+q)}[N^2(2t+p)^2m-2N|2t+p|D+2Nm-D] 
\end{equation}
for $t\in I$.Note that  the function
$$
\mathbb{R} \ni \lambda \mapsto N^2m\lambda^2-2N D\lambda+2Nm-D$$
attains its minimum, equal $-\frac{D^2}{m}-D+2Nm$, at the point $\lambda =\frac{D}{Nm}$. Thus for
$
N>\mathcal{N}(m,D)
$
we have $$N^2m\lambda^2-2ND|\lambda|+2Nm-D>0$$ for any $\lambda\in \mathbb R$. 
Therefore
$$
\varphi_{N,p,q}''(t)\geq -\frac{D^2}{m}-D+2Nm>0\quad \hbox{for }t\in I,
$$
which implies that $\varphi_{N,p,q}$ is strongly convex on $I$.
\end{proof}

From Lemma \ref{lemat dla funkcji2} we immediately obtain

\begin{cor}\label{cor lemat dla funkcji2}
Let $f$ be a function of class $C^2$ on a closed interval $I\subset \mathbb R$. Let 
 $m,D\in\mathbb{R}$ be such that 
\begin{equation}
	m< \inf\{f(t): t \in I\},
\end{equation}
and 
\begin{equation}\label{ograniczenie pochodnychdwa}
|f'(t)|\leq  D, \quad |f''(t)|\leq  D, \quad \hbox{for} \quad t \in I.
\end{equation} 
Than for any $\xi\in\rr$ and any $N\geq 1$ the   function 
$$
\psi_{N,\xi}(t)=e^{N(t-\xi)^2}[f(t)-m+D],\quad t\in I
$$
is strongly convex on $I$. In particular the function
$$
\varphi(t)=e^{(t-\xi)^2}[f(t)-m+D],\quad t\in I,
$$
is strongly convex on $I$.
\end{cor}

\subsection{Convexifying polynomials in several variables}\label{Sect3}

We will show that the function $\varphi _N$ in $n$ variables is strongly convex on a compact convex set $X\subset\mathbb{R}^n$, provided $f$ is a polynomial positive on $X$ and $N$ is suficiently large.

Let 
$f\in \mathbb R [x]$ be a real polynomial in $x=(x_1, \ldots , x_n)$ of the form
\begin{equation}\label{funkcja f}
f=\sum_{j=0}^d\sum_{|\nu |=j}a_{\nu }x^\nu ,
\end{equation}
where $a_\nu  \in \mathbb R,$ $x^\nu =x^{\nu_{1}}_1\cdots x^{\nu_{n}}_n$ and $|\nu |=\nu _1 + \cdots + \nu _n$ for $\nu =(\nu _1, \cdots , \nu _n)\in \mathbb N^n$ (we assume that $0\in \mathbb N$). 
For $R>0$ we denote
$$
D_n (f,R):=\max \bigg\{1,\sum_{j=1}^d\sum_{|\nu |=j}j|a_ \nu |R^{j-1};\sum_{j=1}^d\sum_{|\nu |=j}j(j-1)|a_ \nu |R^{j-2}\bigg\}.
$$

\begin{tw}\label{uwypuklanie na zwartym}
Let $f\in \mathbb R [x]$ be a polynomial which is positive on a compact and convex set $X\subset \mathbb R^n$. Let $R=\max\{|x|: x \in X\}$ and
$$0<m\leq \min\{f(x): x\in X\}.$$
Than for any $\xi\in \rr^n$, any $D \geq D_n (f,R)$ and any real $N> \mathcal{N}(m,D)$ the function $\varphi _{N,\xi}(x):=e^{N|x-\xi|^{2}}f(x)$ is strongly convex on $X$.
\end{tw}

\begin{proof}
 Let 
\begin{equation}\label{eqdefA}
A:=\{(\alpha,\beta) \in \mathbb R^n \times \mathbb R^n: \langle \alpha, \beta \rangle=0,\; |\beta|=1\},
\end{equation}
where $\langle \cdot,\cdot \rangle$ denotes the standard scalar product on $\mathbb R^n.$ Set
$$\gamma _{\alpha, \beta}(t):=\beta t + \alpha, \ \ t\in \mathbb R.$$
Clearly, the family of all curves $\gamma _{\alpha, \beta}$, where $(\alpha,\beta)\in A$ describes all affine lines in $\mathbb R^n.$
Denote by   $B\subset A$ the set of all $(\alpha,\beta)\in A$ such that the line parametrized by $\gamma _{\alpha, \beta}$ intersects the set X. Then $B$ is a compact set and $$B\subset \{(\alpha ,\beta )\in A: |\alpha |\leq R\}.$$
We will prove that for any $(\alpha ,\beta )\in B$ and $N > \mathcal{N}(m,D)$ the function $\varphi_{N,\xi} \circ \gamma _{ \alpha , \beta }$ is strongly convex on 
$$I_{\alpha ,\beta }:=\{t\in \mathbb R: \gamma_{\alpha ,\beta }(t) \in X\}.$$
Because $X$ is a compact and convex set, so $I_{\alpha ,\beta }$ is a closed interval or only one point. 

It is obvious that for $(\alpha;\beta) \in B$ the set $\{t\in \mathbb R: |\gamma _{\alpha ,\beta }(t)|\leq R\}$ is an interval, which contains the point 0 or it is equal to $\{0\}$. Denote  this interval by $[-R_{\alpha ,\beta },R_{\alpha ,\beta }]$ (under convention $[0,0]=\{0\}$). Then 
$$I_{\alpha ,\beta }\subset [-R_{\alpha ,\beta },R_{\alpha ,\beta }]\subset [-R;R].$$

Let $f$ be of the form (\ref{funkcja f}). Than for $t\in I_{\alpha, \beta}$ we have
$|\gamma _{\alpha ,\beta } (t)| \leq R $.
Let us fix $(\alpha, \beta) \in B$. 
Then
$$
|(f\circ \gamma _{\alpha ,\beta })'(t)|\leq \sum_{j=1}^d\sum_{|\nu |=j}j|a_ \nu |R^{j-1}
$$
and
$$
|(f\circ \gamma _{\alpha ,\beta })''(t)|\leq \sum_{j=1}^d\sum_{|\nu |\leq j}j(j-1)|a_ \nu |R^{j-2}.
$$
Consequently, 
$$
|(f\circ \gamma _{\alpha ,\beta })'(t)|\leq D, \ \ \ |(f\circ \gamma _{\alpha ,\beta })''(t)|\leq D \ \ \ \hbox{for} \ \ \ t\in I_{\alpha, \beta}.$$

Take any $\xi\in\rr^n$, then
\begin{multline*}
|\gamma _{\alpha ,\beta }(t)-\xi|^2=\langle \beta t+\alpha-\xi ,\beta t+\alpha-\xi \rangle \\
= t^2-2\langle \beta,\xi\rangle t+|\alpha| ^2-2\langle \alpha,\xi\rangle+|\xi|^2,
\end{multline*}
then for $p=-2\langle \beta,\xi\rangle$ and $q=|\alpha| ^2-2\langle \alpha,\xi\rangle+|\xi|^2$, we have $p^2\leq 4q$ and
$$
\varphi _N \circ \gamma _{\alpha, \beta}(t)=e^{N(t^2+pt+q)}f(\gamma _{\alpha, \beta}(t)).
$$
So, by Lemma {\ref{lemat dla funkcji2}} we get that $\left(\varphi _N \circ \gamma _{\alpha, \beta}\right)''(t)\geq  -\frac{D^2}{m}-D+2Nm>0$ for $t\in I_{\alpha,\beta}$ and $\varphi_{N,\xi}$ is strongly convex on $X$, 
 provided $N>\mathcal{N}(m,D)$.
\end{proof}
From Theorem \ref{uwypuklanie na zwartym}  we obtain the following corollary.

\begin{cor}\label{uwypuklanie na zwartymcor2}
Let $f\in \mathbb R [x]$  and let  $X\subset \mathbb R^n$ be a compact and convex set. Let $R=\max\{|x|: x \in X\}$ and let $m\in \mathbb{R}$ be a constant such that
$$
m\leq \min\{f(x): x\in X\}.
$$
Than for any $D> D_n (f,R)$ and any $\xi\in\rr^n$, the function
$$
\varphi_\xi(x):=e^{|x-\xi|^2}[f(x)-m+D], \quad x\in \mathbb{R}^n
$$
is strongly convex on $X$.
\end{cor}

By a similar argument as in the proof of Theorem \ref{uwypuklanie na zwartym}, we obtain the following fact.

\begin{rem}\label{uwypuklanie na zwartymrem}
Let $f:\mathbb{R}^n\to \mathbb{R}$ be a function of class $C^2$ and let $X\subset \mathbb R^n$ be a compact and convex set. Assume that $m,D\in \mathbb{R}$ are numbers satisfying 
$$
m < \min\{f(x): x\in X\}
$$
and the first and second directional derivatives of $f$ in directions of vectors of length $1$, are bounded by $D$ on $X$. Then the function
$$
\varphi_\xi(x)=e^{|x-\xi|^2}[f(x)-m+D], \quad x\in \mathbb{R}^n, \quad \xi\in \mathbb{R}^n
$$
is strongly convex on $X$. 
\end{rem}

\subsection{Convexifying polynomials with integer coefficients}\label{SectINTEGER}

For actual applications  of  Theorem \ref{uwypuklanie na zwartym} it is important to compute  the number $\mathcal{N}(m,D)$ 
for a given  convex semialgebraic set $X $
and a polynomial $f$ which is positive on $X$. Hence 
the main  difficulty is to compute (or rather estimate)  $m=\min\{f(x):x\in X\}$ and $R=\max\{|x|:x\in X\}$. This actually possible if we suppose that 
$f$ has integer coefficients and $X$ is described by  equations and inequalities  with integer coefficients.

%
 More precisely, let $X\subset \rr^n$, $n\ge 2$, be a compact semialgebraic set of the form 
\begin{multline}\label{formXc}
X=\{x\in\mathbb{R}^n:g_1(x)=0,\ldots,g_l(x)=0,g_{l+1}(x)\geq 0,
\ldots,\\
g_k(x)\geq 0\},
\end{multline}
where $g_1,\ldots,g_k\in \mathbb{Z}[x]$. 
Under the above notations 
G. Jeronimo, D.~Perrucci, E.~Tsigaridas in \cite{JPT} proved that

\begin{tw}\label{uwypuklaniecalkowitychna zwartym}
Let $f, g_1,\ldots,g_k\in \mathbb{Z}[x]$ be polynomials with degrees bound by an even integer $d$ and coefficients of absolute values at most $H$, and let $\tilde H=\max\{H,2n+2k\}$. If $f(x)>0$ for $x\in X$ and $X$ of the form \eqref{formXc} is compact, then 
$$ 
\min\{f(x):x\in X\}\geq \left(2^{4-\frac{n}{2}}\tilde H d^n\right)^{-n2^nd^n}.
$$
\end{tw}

For a positive real number $H$ and positive integers $d,n,k$ we put
$$
\frak{b}(n,d,H,k)=\left(2^{4-\frac{n}{2}}\max\{H,2n+2k\}d^n\right)^{-n2^nd^n}
$$

From Theorems \ref{uwypuklaniecalkowitychna zwartym} and \ref{uwypuklanie na zwartym} we immediately obtain

\begin{tw}\label{uwypuklanie na zwartym calk}
Let $X\subset \rr^n$ be a compact and convex semialgebraic set of the form \eqref{formXc} and let $f, g_1,\ldots,g_k\in \mathbb{Z}[x]$ be polynomials with degrees bound by an even integer $d $ and coefficients of absolute values at most $H$. 
 Set
$$
R=\sqrt{\big[\frak{b}(n+1,\max\{d,4\},H,k+2)\big]^{-1}-1},\quad m=\frak{b}(n.d,H,k).
$$
Then  
\begin{equation}\label{boundR}
\max\{|x|:x\in X\}\leq R.
\end{equation}
Moreover, if $f(x)>0$ for $x\in X$, then for any $D \geq D_n \left(f,R\right)$,  $N> \mathcal{N}\left(m,D\right)$  and for any $\xi\in\rr^n$ 
the function $$\varphi _{N,\xi}(x):=e^{N|x-\xi|^{2}}f(x)$$ is strongly convex on $X$.
\end{tw}

\begin{proof}
By Theorem \ref{uwypuklaniecalkowitychna zwartym} we have $0<m\leq \min\{f(x):x\in X\}$. 
Let
$$
Y=\{(x,y)\in\rr^n\times\rr:x\in X,\,(1+|x|^2)y^2-1=0,\,y\ge 0\},
$$
and let $h(x,y)=y^2$. Then $Y\subset \rr^{n+1}$ is a compact semialgebraic set defined by $k+2$ polynomial equations and inequalities of degrees bounded by $\max\{d,4\}$. Moreover, the absolute values of coefficients of those polynomials and $h$ are bounded by $H$. Then, by Theorem \ref{uwypuklanie na zwartym calk},  
$$
\min\{h(x,y):(x,y)\in Y\}\geq \frak{b}(n+1,\max\{d,4\},H,k+2)
$$
and consequently we obtain \eqref{boundR}. Summing up, Theorem \ref{uwypuklanie na zwartym} gives the assertion. 
\end{proof}

\section{Convexifying polynomials on non-compact sets}\label{Sect44}

In this section we will show that the function $\varphi _N(x)=e^{N|x|^2}f(x)$ in $n$ variables is strongly convex on a closed and convex set $X\subset\mathbb{R}^n$ (not necessary compact), provided the polynomial $f$ takes values larger than a certain number $m>0$, the leading form of a polynomial $f$ has only positive values and $N$ is suficiently large.

\subsection{
 Convexifying polynomials in one variable}

For a polynomial $f\in\rr[t]$ of the form $f(t)=a_0t^d+a_1t^{d-1}+\cdots+a_d$, $a_0,\ldots,a_d\in\rr$, $a_0\ne 0$, we put
$$
K(f):=2\max_{1\leq i\leq d}\left|\frac{a_i}{a_0}\right|^{1/i}.
$$

\begin{lemat}\label{K(g)<K(f)var}
Let $f\in \mathbb{R}[t]$ be a polynomial of degree $d>0$ which is positive on a closed interval $I\subset \rr$ (not necessary compact). Let $m\in\rr$ be a positive number such that
$$
\inf\{f(t):t\in I\}\geq m.
$$
Let $g_N\in\rr[t]$ be a polynomial of the form 
$$
g_N=2Nf^2-(f')^2+ff'',
$$
and let $\Theta_N\in\mathbb{R}[t,\xi]$ be a polynomial of the form
\begin{equation}\label{thetaNformvar}
\Theta _N(t,\xi):=4N^2(t-\xi)^2f(t)+4N(t-\xi)f'(t)+2Nf(t)+f''(t)
\end{equation}
for $N\in \mathbb{R}$ and $N \geq 1$. Then 
for $N\geq \mathcal{N}(m,D)$, where $D\geq D_1(f,R)$ and $R\geq \max\{K(f),K(g_1)\}$, we  have
$$
\Theta_N(t,\xi)>0\quad\hbox{for }(t,\xi)\in I\times \rr.
$$
\end{lemat}

\begin{proof}
Consider the following quadratic function in $x$
$$
4N^2x^2f(t)+4Nxf'(t)+2Nf(t)+f''(t).
$$
Then its discrirminant is of the form $\Delta(t)=-16N^2g_N(t)$. 
Take $R\geq \max\{K(f),K(g_1)\}$, $D\geq D_1(f,R)$ and $N>\mathcal{N}(m,D)$. Then we have 
$$
g_N(t)\geq 2Nf^2(t)-D^2-f(t)D\geq f^2(t)\left(2N-\tfrac{D^2}{m^2}-\tfrac{D}{m}\right)>0
$$
for $t\in I$, $|t|\leq R$. On the other hand $g_N(t)\geq g_1(t)>0$ for $t\in I$, $|t|\geq R$. So $\Delta(t)<0$ for $t\in I$ and we deduce the assertion. 
\end{proof}

\begin{tw}\label{thm3unboundedvar}
Let $f \in \mathbb R[t] $ be a polynomial of degree $d>0$ and let $I\subset \mathbb{R}$ be a closed interval (not necessary compact). Assume that there exists $m\in \mathbb R$ such that 
$$
0<m\leq \inf\{f(t): t\in I\}.
$$ 
Let $R>\max\{K(f),K(2f^2-(f')^2+ff'')$ and $D\geq D_1(f,R)$. 
Then for any $N\in \mathbb{R}$, $N\geq \mathcal{N}(m,D)$, and any $\xi\in\rr$ the function 
$$
\varphi_{N,\xi}(t)=e^{N(t-\xi)^2}f(t)
$$ 
is strongly convex on $I$.
\end{tw}

\begin{proof}
It suffices to observe that $\varphi''_{N,\xi}(t)=e^{N|t-\xi|^2}\Theta_N(t,\xi)$ and apply Lemma \ref{K(g)<K(f)var}.
\end{proof}

\subsection{
Convexifying polynomials in several variables}

\begin{tw}\label{convexonconvexsetnewvar}
Let $X\subset \rr^n$ be a convex closed set. Assume that 
$f$ is a polynomial of degree $d>0$ which is positive on~$X$,
\begin{equation}\label{assumptioncompactvar}
f_d^{-1}(0) = \{0\}
\end{equation}
and there exists $m>0$ such that
\begin{equation}\label{estfcompact1var}
\inf \{f(x):x\in X\} \ge m. 
\end{equation}
Then there exists $N_0\in\nn$ such that for any integer $N\geq N_0$ and any $\xi\in\rr^n$ 
 the function $\vf_{N,\xi}(x)=e^{N|x-\xi|^2}f(x)$ is strongly convex on $X$. 
\end{tw}

\begin{proof}
Take any line of the form $\gamma_{\alpha,\beta}(t)=\beta t +\alpha$, where $\alpha,\beta\in\rr^n$, $|\beta|=1$ and $\langle \alpha,\beta\rangle=0$. Then
$$
(\varphi_{N,\xi}\circ \gamma_{\alpha,\beta})(t)=e^{N(t^2+|\alpha|^2-2\langle \beta,\xi\rangle t-2\langle\alpha,\xi\rangle)}f(\gamma_{\alpha,\beta}(t)).
$$
Then
\begin{equation*}\label{eqbis}
\begin{split}
(\varphi_{N,\xi}\circ \gamma_{\alpha,\beta})''(t)=&e^{N(t^2+|\alpha|^2-2\langle \beta,\xi\rangle t-2\langle\alpha,\xi\rangle)}[4N^2(f\circ\gamma_{\alpha,\beta})(t)y^2\\
&+4N(f\circ\gamma_{\alpha,\beta})'(t)y+ 2N(f\circ\gamma_{\alpha,\beta})(t)\\
&+(f\circ\gamma_{\alpha,\beta})''(t)],
\end{split}
\end{equation*}
where $y=t+\langle\beta,\xi\rangle$. Consider the function in the square bracket as a quadratic function in $y$. Then its discriminant is of the form 
$$
\Delta=-16N^2[2N(f\circ\gamma_{\alpha,\beta})^2(t)+(f\circ\gamma_{\alpha,\beta})(t)(f\circ\gamma_{\alpha,\beta})''(t)-((f\circ\gamma_{\alpha,\beta})'(t))^2].
$$
Note that $(f\circ\gamma_{\alpha,\beta})'(t)$ and $(f\circ\gamma_{\alpha,\beta})''(t)$ are the first and the second directional derivatives of $f$ at $\gamma_{\alpha,\beta}(t)$ in the direction $\beta$ and $|\beta|=1$. 

Observe that there exists $N_0$ such that for any $N\geq N_0$ we have $\Delta<0$. Indead, it suffices to prove that for any $x\in X$ and any $\beta \in\rr^n$, $|\beta|=1$ we have
\begin{equation}\label{eqvar1}
2Nf(x)^2+f(x)\partial^2_\beta f(x)-(\partial_\beta f(x))^2>0.
\end{equation}
If $f_d(x)<0$ for $x\in\rr^n\setminus \{0\}$ then the set $X$ is compact and the inequality follows from the assumption that $f(x)\geq m$ for $x\in X$. Indead, let $D\geq \max \{|\partial _\beta f(x)|,|\partial^2_\beta f(x)|\}$ for $x\in X$, $|\beta|=1$. Since $X$ is compact, then
$$
2Nf(x)^2+f(x)\partial^2_\beta f(x)-(\partial_\beta f(x))^2\ge 2Nm^2-mD-D^2>0
$$
for $N>\mathcal{N}(m,D)$. This gives \eqref{eqvar1}.

Consider the case when $f_d(x)>0$ for $x\in\rr^n\setminus\{0\}$, and let 
$$
f_{d*}=\inf \{f_d(x):x\in S_n\},
$$
where $S_n$ is the unit sphere in $\mathbb{R}^n$, i.e., $S_n=\{x\in\mathbb{R}^n:|x|=1\}$.

Let $f$ be a polynomial of the form \eqref{funkcja f}. We set 
$$
\|f\|:=
\sum_{|\nu|\leq d}|a_\nu|.
$$
Then $\|f\|\geq \|f_d\|\geq f_{d*}$. If $f_{d*}>0$ then  we set
$$
\mathbb{K}(f):= \frac{2\|f\|}{f_{d*}}
$$
and 
$$
m(f):=f_{d*}-\sum_{j=0}^{d-1}\mathbb{K}(f)^{j-d}\sum_{|\nu|=j}|a_\nu|.
$$

In the further part of the proof we will need the following lemma.

\begin{lemat}\label{lemmamestvar}
If $d=\deg f>0$ and $f_{d*}>0$, then $m(f)>0$ and $f(x)\geq m(f)|x|^d$ for any $x\in \mathbb{R}^n$ such that $|x|\geq \mathbb{K}(f)$.
\end{lemat}

\begin{proof} 
Put
$$
h(t):=f_{d*}t^d-\sum_{j=0}^{d-1}\left(\sum_{|\nu|=j}|a_\nu|\right)t^{j}.
$$
Since $\frac{\|f\|}{f_{d*}}\geq 1$, then 
$$
K(h)=2\max_{1\leq i\leq d}\left|\frac{\sum_{|\nu|=d-i}|a_\nu|}{f_{d*}}\right|^{{1}/{i}}<  2\max_{1\leq i\leq d}\left|\frac{\|f\|}{f_{d*}}\right|^{{1}/{i}}=\mathbb{K}(f),
$$
and since $h'(t)>0$ for $t >K(h)$, then $h(|x|)\geq h(\mathbb{K}(f))>0$ for $|x|\geq \mathbb{K}(f)$. Moreover, $m(f)\mathbb{K}(f)^d=h(\mathbb{K}(f))$, so $m(f)>0$. On the other hand
$$
m(f)|x|^d\leq \left(f_{d*}-\sum_{j=0}^{d-1}|x|^{j-d}\sum_{|\nu|=j}|a_\nu|\right)|x|^d=h(|x|)\leq f(x)
$$
for $|x|\geq \mathbb{K}(f)$. This gives the assertion of Lemma \ref{lemmamestvar}.
\end{proof}

Take $R\geq \mathbb{K}(f)$, and $D\geq D_n(f,R)$ then for $N\geq \mathcal{N}(m,D)$ we have
$$
2Nf(x)^2+f(x)\partial^2_\beta f(x)-(\partial_\beta f(x))^2\geq 2Nm^2-mD-D^2>0
$$
for $|x|\leq R$. 

For $|x|\geq R$, we have
\begin{multline*}
2Nf(x)^2+f(x)\partial^2_\beta f(x)-(\partial_\beta f(x))^2\\
\geq 2Nm^2(f)|x|^{2d}-m(f)|x|^dD_n(f,|x|)-D_n^2(f,|x|).
\end{multline*}
Since for $|x|\geq 1$,
$$
D_n(f,|x|)\leq D_n(f,1)|x|^{d-1},
$$
then for $|x|\geq R$ and $|\beta|=1$ we have
\begin{multline*}
2Nf(x)^2+f(x)\partial^2_\beta f(x)-(\partial_\beta f(x))^2\\
\geq 2Nm^2(f)|x|^{2d}-m(f)D_n(f,1)|x|^{2d-1}-D_n^2(f,1)|x|^{2d-2}\\
\geq |x|^{2d}[2Nm^2(f)-m(f)D_n(f,1)-D_n^2(f,1)]>0.
\end{multline*}
for $N>\mathcal{N}(m(f),D_n(f,1))$. This gives \eqref{eqvar1}. Moreover, there exists $\epsilon>0$ such that 
$$
2Nf(x)^2+f(x)\partial^2_\beta f(x)-(\partial_\beta f(x))^2>\epsilon.
$$
for any $x\in X$ and $|\beta|=1$.

From \eqref{eqvar1} and the above there follows that $\Delta<-16N^2\epsilon$ for any $\alpha,\beta$ and $t$ such that $\gamma_{\alpha,\beta}(t)\in X$. Since $f(x)\geq m$ for $x\in X$, then $(\varphi_{N,\xi}\circ \gamma_{\alpha,\beta})''(t)\geq \tfrac{\epsilon}{m}$ if $\gamma_{\alpha,\beta}(t)\in X$. This gives the assertion.
\end{proof}

By analogous argument as for Theorem \ref{convexonconvexsetnewvar} and under notations of the proof  we obtain the following corollary.

\begin{cor}\label{uwypuklanie na wypuklymcor2}
Let $f\in \mathbb R [x]$ be a polynomial of degree $d$  and let  $X\subset \mathbb R^n$ be a convex and closed set. 
Under assumptions of Theorem \ref{convexonconvexsetnewvar} and notations of the proof,
for any $R\geq \mathbb{K}(f)$, $D\geq D_n(f,R)$ and 
$N>\max\{\mathcal{N}(m,D),\mathcal{N}(m(f),D_n(f,1))\}$,  and any $\xi\in\rr^n$ 
 the function $\vf_{N,\xi}(x)=e^{N|x-\xi|^2}f(x)$ is strongly convex on $X$. In particular 
 the function
$$
\varphi_\xi(x):=e^{|x-\xi|^2}[f(x)-m+D]
$$
is strongly convex on $X$.
\end{cor}

The assumption \eqref{assumptioncompactvar} 
  that $f_d(x)\ne 0$ for $x\ne 0$, in Theorem \ref{convexonconvexsetnewvar},  
  can not be omited as the following example shows.

\begin{exa}\label{exacounter}
Let $f\in \rr[x,y,z]$ be a polynomial of the form
$$
f(x,y,z)=(y^2+z^2+1)\left[(x-1)^2(x+1)^2 +(yz+1)^2+y^2\right].
$$
Since $(y^2+z^2+1)\left[(yz+1)^2+y^2\right]\ge \frac{1}{2}$ for $(y,z)\in\rr^2$ then we easily see that 
$$
f(x,y,z)\ge \tfrac{1}{2}\quad \hbox{for }(x,y,z)\in\rr^3.
$$
Note that $\deg f=6$, and the leading form $f_6(x,y,z)=(y^2+z^2)(x^4+y^2z^2)$ has nontrivial zeroes.

Now take any $N\in \rr$ and $\varphi_N(x,y,z)=e^{N(x^2+y^2+z^2)}f(x,y,z)$.
Then for $\xi\ne 0$ we have
$$
\varphi_N(0,\xi^{-1},-\xi)=e^{N(\xi^{-2}+\xi^2)}(\xi^{-2}+\xi^2+1)(1+\xi^{-2})
$$
and
$$
\varphi_N(-1,\xi^{-1},-\xi)=\varphi_N(1,\xi^{-1},-\xi)=e^{N(\xi^{-2}+\xi^2)}(\xi^{-2}+\xi^2+1)e^N\xi^{-2}.
$$
Hence for sufficiently large $\xi$, 
$$
\varphi_N(-1,\xi^{-1},-\xi)=\varphi_N(1,\xi^{-1},-\xi)<\varphi_N(0,\xi^{-1},-\xi),
$$
therefore $\varphi_N$ can  not be a convex function.
\end{exa}

The assumption \eqref{assumptioncompactvar} in Theorem \ref{convexonconvexsetnewvar}, cannot be  replaced by a condition $\lim_{|x|\to\infty} f(x)=\infty$. 
Indead, consider a modification of the previous example of the form
$$
f_k(x,y,z)=(y^2+z^2+1)^k\left[(x-1)^2(x+1)^2 +(yz+1)^2+y^2\right],
$$
where $k\geq 2$. Then $\lim_{|(x,y,z)|\to\infty} f(x,y,z)=\infty$ and the function $\varphi_N(x,y,z)=e^{N(x^2+y^2+z^2)}f_k(x,y,z)$ is not convex for any $N\in\rr$ by the previous argument.

It turns out that the use of a double exponential function leads to a convexity of an appropriate function on $X$. 
We show it in the next section. 

\section{Double exponential convexifying polynomials}\label{secdoubleexp}

In this section, without the assumption that the leading form of a polynomial $f\in\rr[x]$ in $n$ variables has only positive values, we will show that the function $\varPhi _N(x)=e^{e^{N|x|^2}}f(x)$  is strongly convex on a closed and convex semialgebraic set $X\subset\mathbb{R}^n$ (not necessary compact), provided the polynomial $f$ takes positive values on $X$ and $N$ is suficiently large.

\begin{tw}\label{tedoubleexp}
Let $X\subset \rr^n$ be a closed and convex semialgebraic set, and let $f\in\rr[x]$ be a polynomial which has only positive values on $X$. Then there exists $N_0\in\rr$ such that for any $N\geq N_0$ the function $\varPhi _N(x)=e^{e^{N|x|^2}}f(x)$ is strongly convex on $X$.
\end{tw}

\begin{proof}
Let $f$ be of the form \eqref{funkcja f}, $d=\deg f$. Then $f(x)$, 
, the first and second directional derivatives of $f$ in directions of vectors of length $1$ at $x\in X$, are bounded by
 $\tilde D(1+|x|^d)$, where
$$
\tilde D:=|a_0|+\sum_{|\nu|=1}|a_\nu|+\sum_{j=2}^d\sum_{|\nu |=j}j(j-1)|a_ \nu |.
$$

Take an affine line in $\mathbb R^n$ of the form
$$
\gamma _{\alpha, \beta}(t):=\beta t + \alpha, \ \ t\in \mathbb{R}, 
$$
where $(\alpha,\beta)\in A$ and the set $A$ is defined in \eqref{eqdefA}. Then $|\beta|=1$, $\langle \alpha,\beta\rangle=0$, and $|\gamma_{\alpha,\beta}(t)|^2=t^2+|\alpha|^2$. 
Let write the second derivative of $\varPhi_N\circ \gamma_{\alpha,\beta}$ in the form
$$
(\varPhi_N\circ\gamma_{\alpha,\beta})''(t)=e^{e^{N(t^2+|\alpha|^2)}}(a(t) t^2 +b(t)t +c(t)),
$$
where 
\begin{equation*}
\begin{split}
  a(t) & = 4N^2(e^{N(t^2+|\alpha|^2)}+e^{2N(t^2+|\alpha|^2)})f\circ\gamma_{\alpha,\beta}(t),
  \\ 
b(t) & =  4Ne^{N(t^2+|\alpha|^2)}(f\circ\gamma_{\alpha,\beta})'(t),
\\ 
c(t) & = 2Ne^{N(t^2+|\alpha|^2)}f\circ\gamma_{\alpha,\beta}(t)
+(f\circ\gamma_{\alpha,\beta})''(t)
\end{split}
\end{equation*}
The discriminant of the polynomial 
$
P_t(\lambda) = a(t)\lambda^2 + b(t)\lambda+c(t)
$ 
is of the form 
\begin{equation*}
\begin{split}
\Delta=16N^2e^{2N(t^2+|\alpha|^2)}
&\left[\left((f\circ\gamma_{\alpha,\beta})'(t)\right)^2\right.
\\
&-f\circ\gamma_{\alpha,\beta}(t)(f\circ\gamma_{\alpha,\beta})''(t)\left(1- e^{-N(t^2+|\alpha|^2)}\right)\\
&\left.-2N(f\circ\gamma_{\alpha,\beta})^2(t)\left(1+ e^{N(t^2+|\alpha|^2)}\right)\right].
\end{split}
\end{equation*}
So, by the choice of the 
number $\tilde D$, we have
\begin{equation*}
\begin{split}
\Delta\leq 32N^2e^{2N(t^2+|\alpha|^2)}&\left[ \tilde D^2\left(1+|\gamma_{\alpha,\beta}(t)|^d\right)^2\right.\\
&\;\; \left.-N(f\circ\gamma_{\alpha,\beta})^2(t)\left(1+ e^{N|\gamma_{\alpha,\beta}(t)|^2}\right)
\right].
\end{split}
\end{equation*}

Since the set $X$ is semialgebraic and $f^{-1}(0)\cap X=\emptyset$, then by H\"ormander-{\L}ojasiewicz inequality, see eg. \cite[Corollary 2.4]{KS0}, there exist $C,K,\mathcal{L}>0$, where $K,\mathcal{L}\in \mathbb{Z}$, $K\geq d$, depend on $d$ and the {\it complexity}  of $X$, (i.e., degrees and the number of polynomials describing $X$)   such that 
$$
f(x)\geq C\left({1+|x|^K}\right)^{-\mathcal{L}}\quad \hbox{for }x\in X.
$$
Moreover, the numbers $K,\mathcal{L}$ are effectively computable. 
By the above,
\begin{equation*}
\begin{split}
\Delta\leq 32N^2e^{2N(t^2+|\alpha|^2)}\left({1+|\gamma_{\alpha,\beta}(t)|^K}\right)^{-\mathcal{L}}&\left[ \tilde D^2\left(2+|\gamma_{\alpha,\beta}(t)|^K\right)^{\mathcal{L}+2}\right.
\\
&\left.\;\; -NC^2
\left(1+ e^{N|\gamma_{\alpha,\beta}(t)|^2}\right)\right].
\end{split}
\end{equation*}
If $N$ is large enough, then for   any $x\in\rr^n$ we have $$\tilde D^2(2+|x|^K)^{\mathcal{L}+2}<NC^2(1+e^{N|x|^2}).$$
Therefore  $\Delta<0$, so   
$P_t(\lambda) >0$ for any $\lambda \in \rr$.
Consequently 
$$
(\varPhi_N\circ\gamma_{\alpha,\beta})''(t)=e^{e^{N(t^2+|\alpha|^2)}}P_t(t) >0,
$$
 for $t\in \rr$. Note that  $\lim_{|t|\to \infty}(\varPhi_N\circ \gamma_{\alpha,\beta})''(t)=+\infty$, hence  there exists $\mu>0$ such that $(\varPhi_N\circ\gamma_{\alpha,\beta})''(t)\geq \mu $ for $t\in\rr$. Moreover, the number $\mu$  can  be chosen independet of $\gamma_{\alpha,\beta}$. This gives the assertion.
\end{proof}

\begin{rem}\label{remeffectcomputed}
The number $N_0$ in Theorem \ref{tedoubleexp} can be effectively computed, provided we can estimate the constant $C$. More precisely, under notations in the proof, if $k > (\mathcal{L}+2)K$, then for $|x|\geq 1$ we have
$$
NC^2\left(1+e^{N|x|^2}\right)\geq NC^2+\sum_{j=0}^k\frac{C^2N^{j+1}}{j!}|x|^{2j}>\tilde D^2\left(2+|x|^K\right)^{\mathcal{L}+2}
$$
for
$$
N>k!\max_{i=0,\ldots,\mathcal{L}+2}\left({\tilde D^2C^{-2}2^{\mathcal{L}+2-i}\binom{\mathcal{L}+2}{i}}\right).
$$
If additionally $N\geq \tilde D^2 C^{-2}3^{\mathcal{L}+2}$, then the above inequality holds for any
 $x\in\rr^n$.
\end{rem}

\begin{rem}\label{remdoubleexp}
We cannot omit the assumption in Theorem \ref{tedoubleexp} that the set $X$ is semialgebraic. For instance if $f(x,y)=-y^2+y$ and $X=\{(x,y)\in\rr^2:e^{-e^x}\leq y\leq \tfrac{1}{2},\; x\geq 0\}$, then $f(x,y)>0$ on $X$, but the function $\varPhi_N(x,y)$ is not convex on $X$ for any $N\in \rr$. 
\end{rem}

Assuming that $f(x)\geq m$ on $X$, for some $m>0$, we can omit the assumption in Theorem \ref{tedoubleexp} on semialgebraicity of $X$. More precisely, by a similar argument as in the proof of Theorem \ref{tedoubleexp} we obtain

\begin{tw}\label{tedoubleexp2}
Let $X\subset \rr^n$ be a closed and convex set, and let $f\in\rr[x]$ be a polynomial such that $f(x)\geq m$ for $x\in X$ and some $m>0$. Then there exists $N_0\in\rr$ such that for any $N\geq N_0$ the function $\varPhi _N(x)=e^{e^{N|x|^2}}f(x)$ is strongly convex on $X$.
\end{tw}

\begin{rem}\label{tedoubleexp2bis}
By a similar argument as for the proof of Theorem \ref{tedoubleexp2} we obtain that the assertion of this Theorem occurs not only for the function $\varPhi_{N}$ but also for the function $\Phi_{N}(x)=e^{Ne^{|x|^2}}f(x)$. More precisely, we have:

Let $X\subset \rr^n$ be a closed and convex set, and let $f\in\rr[x]$ be a polynomial such that $f(x)\geq m$ for $x\in X$ and some $m>0$. Then there exists $N_0\in\rr$ such that for any $N\geq N_0$ the function $\Phi _{N}$ 
 is strongly convex on $X$.
\end{rem}

A similar argument as for Theorem \ref{tedoubleexp} gives the following theorems.

\begin{tw}\label{convexonconvexsetnewxdoublevarsemi}
Let $X\subset \rr^n$ be a convex closed semialgebraic set and let $r>0$. If $f$ is a polynomial such that  
\begin{equation}\label{estfcompact1xdoublevarsemi}
f(x)>0\quad\hbox{for }x\in X,
\end{equation}
then there exists $N_0\in\nn$ such that for any integer $N\geq N_0$ and any $\xi\in\rr^n$, $|\xi|\leq r$, 
 the function $\varPhi_{N,\xi}(x)=e^{e^{N|x-\xi|^2}}f(x)$ is strongly convex on $X$. 
 
Moreover, there exists $\alpha\in\rr$ such that the function
$$
\varPhi_\xi(x)=e^{e^{|x-\xi|^2}}[f(x)+\alpha], \quad x\in \mathbb{R}^n 
$$
is strongly convex on $X$, provided $ \xi\in \mathbb{R}^n$, $|\xi|\leq r$. 
\end{tw}


\begin{tw}\label{convexonconvexsetnewxdouble}
Let $X\subset \rr^n$ be a convex closed set and let $r>0$. Assume that 
$f$ is a polynomial of degre $d>0$ such that
 there exists $m\in\rr$ such that
\begin{equation}\label{estfcompact1xdouble}
0<m < \inf \{f(x):x\in X\}.
\end{equation}
Then there exists $N_0\in\nn$ such that for any integer $N\geq N_0$ and any $\xi\in\rr^n$, $|\xi|\leq r$
 the function $\varPhi_{N,\xi}(x)=e^{e^{N|x-\xi|^2}}f(x)$ is strongly convex on $X$. 
 
Moreover, there exists $\alpha\in\rr$ such that the function
$$
\varPhi_\xi(x)=e^{e^{|x-\xi|^2}}[f(x)+\alpha], \quad x\in \mathbb{R}^n 
$$
is strongly convex on $X$, provided $ \xi\in \mathbb{R}^n$, $|\xi|\leq r$. 
\end{tw}

It is impossible to obtain $N$ in the above theorem, such that the function $\varPhi_{N,\xi}$ is convex for any $\xi\in X$ as the following example shows.

\begin{exa}\label{exacounter2}
Let $f\in \rr[x,y,z]$ be a polynomial of the form
$$
f(x,y,z)=\left[(yz+1)^2+y^2\right]\left[\left(xz^2-1\right)^2\left(xz^2+1\right)^2 +y^2+z^2+1\right].
$$
Analogously as in Example \ref{exacounter} 
 we see that 
$$
f(x,y,z)\ge \tfrac{1}{2}\quad \hbox{for }(x,y,z)\in\rr^3,
$$
and the leading form $f_{16}(x,y,z)=y^2z^{10}x^4$ has nontrivial zeroes.

Now take any $N\in \rr$ and $\varPhi_N(x,y,z)=e^{e^{N(x^2+y^2+z^2)}}f(x,y,z)$.
Then for $\xi=(0,t^{-1},-t)$, $t>0$ we have
$$
\frac{\partial^2 \varPhi_{N,\xi}}{\partial x^2}(\xi)=e\left(2Nf(\xi)+\frac{\partial^2 f}{\partial x^2}(\xi)\right).
$$
Since 
$$
f(\xi)=2t^{-2}+t^{-4}+1
$$
and
$$
\frac{\partial^2 f}{\partial x^2}(\xi)=-4t^2,
$$
then 
we easily see that $\frac{\partial^2 \varPhi_{N,\xi}}{\partial x^2}(\xi)<0$ for sufficiently large $t$. So, $\varPhi_{N,\xi}$ can not be a convex function.
\end{exa}

\begin{rem}\label{remtripleexpmotnotimproves}
It is worth noting that the use of triple exponential convexifying $\phi_N(x)=e^{e^{e^{N|x|^2}}}f(x)$ of a polynomial $f$ does not improve convexity of the function  
  $\phi_{N,\xi}(x)=e^{e^{e^{N|x-\xi|^2}}}f(x)$ regardless of $\xi\in X$.
\end{rem}

\section{Algorithm for searching lower critical points}\label{Algorithm0}

\subsection{Searching lower critical points in a compact set}

In this part we give an algorithm which produces, starting from an arbitrary point, a sequence of points converging to a lower critical point of a polynomial on a convex compact semialgebraic set. A similar algorithm was proposed in \cite{KS}.

Let $X\subset \mathbb R^n$ be a closed set and let $f$ be a function of class $C^1$ in a neighborhood $U\subset \mathbb R^n$ of  $X$. 
We denote the set of lower critical points of the function $f$ on the set $X$ by $\Sigma _X f.$ It is obvious that the set of ordinary critical points $\Sigma f$ of the function $f$ is contained in the set $\Sigma _X f.$ 

Our algorithm for approximation of lower critical points  of $f$ is based on the iteration of computation of the smallest value of the strongly convex function $\varphi _{\xi }$ on the convex and compact set $X$. More precisely, let 
$$
R\geq
\max\{|x|:x\in X\}.
$$
 Take any polynomial $f\in\mathbb{R}[x]$ of the form \eqref{funkcja f}. Let 
$$
m=-\sum_{j=0}^dR^j\sum_{|\nu |=j}|a_{\nu }|,
$$
and let 
$$
D>D_n (f,2R).
$$
Then we have  
$$
f(x)-m+D\geq D \quad\hbox{for }x\in X, 
$$
and from Corollary \ref{uwypuklanie na zwartymcor2}, we have that for any $\xi\in X$, the function 
$$
\varphi_\xi(x)=e^{|x-\xi|^2}[f(x)-m+D],\quad x\in \mathbb{R}^n
$$
is $\mu$-strongly convex on $X$ for some $\mu>0$. Since we are looking for lower critical points of $ f $, so without loss of generality, we may assume that $-m+D=0$, therefore
$$
\varphi_\xi(x)=e^{|x-\xi|^2}f(x),\quad x\in \mathbb{R}^n
$$
is $\mu$-strongly convex function in $X$ for any $\xi\in X$.

Any strictly convex function $\varphi$ defined on a compact and convex set $X$ has the unique point, denoted by $\operatorname{argmin}_X \varphi$, in which the function $\varphi$ has the minimal value on the set $X$. Therefore, chosing an arbitrary point $a_0 \in X$, we can determine by induction a sequence $a_\nu \in X$, $\nu\in\mathbb{N}$, in the following way
\begin{equation}\label{eqanu}
a_\nu :=\operatorname{argmin}_X \varphi_{a_{\nu -1}}\quad \hbox{for }\nu \geq 1.
\end{equation}


\begin{tw}\label{approxi}
Let $X\subset \mathbb R^n$ be a compact convex semialgebraic set and $f:\mathbb R ^n \rightarrow \mathbb R$ a positive polynomial on $X$. Let $a_\nu $ be a sequence defined as $a_ \nu :=\operatorname{argmin}_X \varphi_{a_{\nu -1}}$ with $a_0 \in X.$ Then the limit $$
a_*=\lim_{\nu \rightarrow \infty } a_ \nu 
$$ 
exist and $a_* \in \Sigma _X f.$
\end{tw}

The proof of Theorem \ref{approxi} follows word by word  
 the proof of Theorem 6.5 in \cite{KS}, where we should use the following  three lemmas instead of the corresponding  lemmas in \cite{KS}. 


\begin{lemat}
For any $\nu \in \mathbb N$, we have
$$|a_{\nu +1}-a_{\nu }|=\operatorname{dist}(a_\nu , f^{-1}(f(a_{\nu +1})) \cap X).$$
\end{lemat}

\begin{lemat}\label{lemma44x}
For any $\nu \in \mathbb N$ we have
$$f(a_{\nu +1})\leq \frac{f(a_\nu )-\frac{\mu }{2}|a_{\nu +1}-a_\nu |^2}{e^{|a_{\nu +1}-a_\nu |^2}}.$$
In particular the sequence $f(a_\nu )$ is decreasing.
\end{lemat}
\begin{proof}
Since $\varphi _\xi $ is strongly convex, the definition of $a_{\nu +1}$ implies that the function
$$
[0,1] \ni t\mapsto \varphi _{a_\nu }(a_\nu +t(a_{\nu +1}-a_\nu ))
$$
decrease, so $\langle a_{\nu +1}-a_{\nu }, \nabla \varphi _{a_{\nu }}(a_{\nu +1 })\rangle \leq 0$.
Again by the fact  that  $\varphi _{a_\nu }$ is $\mu$-strictly convex, we get 
$$
f({a_\nu })\geq f(a_{\nu +1})e^{|a_{\nu }-a_{\nu +1}|^2}+\frac{\mu}{2}|a_{\nu }-a_{\nu +1} |.$$
This gives the assertion.
\end{proof}

We can also addapt the following lemma (\cite[Lemma 6.3]{KS}).
\begin{lemat}
Let $f:[0,\eta ]\rightarrow \mathbb R$ be a $C^1$ function such that $0<f\leq C$ and $f'\leq -\eta $ on $[0,\eta ]$ for some $C\geq \frac{1}{2}$ and $\eta >0.$ 
Assume that $\varphi (x)=e^{x^2}f(x)$ is strictly convex on $[0,\eta ].$ Then $b_1:=\operatorname{argmin}_{[0,\eta ]}\varphi \geq \frac{\eta }{2C}.$ Hence $f(0)-f(b_1)\geq \frac{\eta ^2}{2C}.$
\end{lemat}

%

\begin{rem}\label{effectivemin}
The function $\varphi_{a_{\nu-1}}$ is defined by using the function $\exp$. 
However, to determine the minimum value of this function on a compact convex semialgebraic set $X$ it is enough to solve only polynomial equations and inequalities. More precisely, the set $X$ is the union of a finite collection of basic semialgebraic sets, so we may assume that
$$
a_\nu\in X=\{x\in\mathbb{R}^n:g_1(x)\geq 0,\ldots,g_k(x)\geq 0\},
$$  
where $g_1,\ldots,g_k\in\mathbb{R}[x]$. Then
$$
 X=\{x\in\mathbb{R}^n:g_1(x)e^{|x-a_{\nu-1}|^2}\geq 0,\ldots,g_k(x)e^{|x-a_{\nu-1}|^2}\geq 0\}.
$$
Therefore, when applying Lagrange Multipliers or Karush-Kuhn-Tucker Theorem to compute the point $a_\nu$ it is enought to solve a system of polynomial equations and inequalities.
\end{rem}

\subsection{Searching lower critical points in an unbounded set}

Let $X\subset \rr^n$ be a convex and closed semialgebraic set. Let $f\in \rr[x]$ be a polynomial of degree $d>0$ of the form \eqref{funkcja f} and let $f_d$ be the leading form of $f$. Assume that $f_{d*} >0$. 

Then by Theorem \ref{convexonconvexsetnewvar}, we may effectively compute a real number $N\geq 1$ such that the function $\varphi_{N,\xi}(x)=e^{N|x-\xi|^2}f(x)$ for $\xi\in\rr^n$ is strongly convex on $X$. Moreover, $\varphi_{N,\xi}(x)\geq f(x)\geq m(f)|x|^d$ for $x\in X$, $|x|\geq \mathbb{K}(f)$, so we have
\begin{equation*}\label{factlim}
\lim_{x\in X,\,|x|\to\infty} \varphi_{N,\xi}(x)=+\infty.
\end{equation*}
Then we may uniquely determine the sequence 
\begin{equation}\label{eqanu2}
a_\nu :=\operatorname{argmin}_{X} \varphi_{a_{\nu -1}}\quad \hbox{for }\nu \geq 1.
\end{equation}
Analogous argument  as for Theorem \ref{approxi} gives the following theorem.

\begin{tw}\label{approxi2}
Let $a_\nu $ be a sequence defined by \eqref{eqanu2}, 
Then the limit $$
a_*=\lim_{\nu \rightarrow \infty } a_ \nu 
$$ 
exist and $a_* \in \Sigma _X f.$
\end{tw}

\begin{rem}\label{generalcase}
If $X\subset \rr^n$ is a closed and convex semialgebraic set and a polynomial $f\in\rr[x]$ is positive on $X$ and it is proper on $X$ (i.e., $\lim_{x\in X,\,|x|\to \infty}f(x)=+\infty$), then by Theorem \ref{convexonconvexsetnewxdoublevarsemi} one can repeat the argument from Theorem \ref{approxi} and obtain a sequence $a_\nu\in X$ such that $\lim_{\nu\to\infty}a_\nu=a_*\in\Sigma_Xf$.

If we assume only that $f(x)>0$ on $X$, 
 then the sequence $a_\nu$ can tend to infinity. Moreover, in the construction of $a_\nu$ we have to change $N$ step by step.
\end{rem}

\end{document}